\documentclass[reqno, 12pt,a4letter]{amsart}
\usepackage{amsmath,amsxtra,amssymb,latexsym, amscd,amsthm}
\usepackage{graphicx,color}
\usepackage[utf8]{inputenc}
\usepackage[mathscr]{euscript}
\usepackage{mathrsfs}
\usepackage[english]{babel}
\usepackage{color}
\usepackage{amsmath,amssymb,amsthm}
\usepackage{graphicx}
\usepackage{hyperref}

\newtheorem {theorem}{Theorem}[section]
\newtheorem {corollary}{Corollary}[section]
\newtheorem {proposition}{Proposition}[section]
\newtheorem {lemma}{Lemma}[section]
\newtheorem {example}{Example}[section]
\newtheorem {definition}{Definition}[section]
\newtheorem {remark}{Remark}[section]

\def\ex{\begin{example}
  }
  \def\eex{\end{example}}
\def\thr{\begin{theorem}}
\def\ethr{\end{theorem}}
\def\pro{\begin{proposition}}
\def\epro{\end{proposition}}
\def\coro{\begin{corollary}}
\def\ecoro{\end{corollary}}
\def\df{\begin{definition}}
\def\edf{\end{definition}}
\def\lm{\begin{lemma}}
\def\elm{\end{lemma}}
\def\pf{\begin{proof}}
\def\epf{\end{proof}}
\def\problem{\begin{problem}}
\def\eproblem{\end{problem}}

\def\ite{\begin{itemize}}
\def\hite{\end{itemize}}
\def\rem{\begin{remark}}
\def\erem{\end{remark}}

\def\cla{\begin{claim}}
\def\ecla{\end{claim}}

\def\Ker{\mbox{Ker}}
\newcommand{\seq}[1]{\left<#1\right>}

\def\ees{{\accent"5E e}\kern-.385em\raise.2ex\hbox{\char'23}\kern-.08em}
\def\EES{{\accent"5E E}\kern-.5em\raise.8ex\hbox{\char'23 }}
\def\ow{o\kern-.42em\raise.82ex\hbox{
\vrule width .12em height .0ex depth .075ex \kern-0.16em \char'56}\kern-.07em}
\def\OW{O\kern-.460em\raise1.36ex\hbox{
\vrule width .13em height .0ex depth .075ex \kern-0.16em \char'56}\kern-.07em}
\def\uw{u\kern-.460em\raise1.36ex\hbox{
\vrule width .13em height .0ex depth .075ex \kern-0.16em \char'56}\kern-.07em}
\title[]{HANDELMAN'S POSITIVSTELLENSATZ FOR POLYNOMIAL MATRICES POSITIVE DEFINITE ON POLYHEDRA}

\author{L\^{e} C\^{o}ng-Tr\`{i}nh $^\dagger$}
\address{Department of Mathematics, Quy Nhon University, 170 An Duong Vuong, Quy Nhon, Binh Dinh, Vietnam}
\email{lecongtrinh@qnu.edu.vn}

\author{D\uw ~Th\d{i}-H\`{o}a-B\`{i}nh}
\address{Department of Natural Sciences, Ha Tay College of Pedagogy, Ha Noi, Vietnam}

\email{hoabinhcdsp@gmail.com}
 


\date{ \today}

\begin{document}
\maketitle

\begin{abstract} 
In this paper we give a matrix version of Handelman's Positivstellensatz \cite{Ha}, representing polynomial matrices which are   positive definite on convex, compact polyhedra.   Moreover, we propose also  a procedure to find such a  representation. As a corollary of Handelman's theorem, we give a special case of Schm\"udgen's Positivstellensatz for polynomial matrices positive definite on  convex, compact polyhedra.
\end{abstract}
\tableofcontents
\section{Introduction}

Let $\mathbb R[X]:=\mathbb R[X_1,\cdots,X_n]$ be the ring of polynomials in the variables $X_1,\cdots,X_n$ with real coefficients.  Denote by $\Delta_n$ the standard $n$-simplex in $\mathbb R^n$, which is defined by 
$$ \Delta_n:=\{x=(x_1,\cdots,x_n)\in \mathbb R^n| x_i\geq 0, \sum_{i=1}^n x_i=1\}.$$

P\'{o}lya \cite{Po} proved in 1928 that for a homogeneous  polynomial $f\in \mathbb R[X]$, \textit{if $f(x)>0$ for every $x\in \Delta_n$, then there exists a sufficiently large number $N$ such that all coefficients of the polynomial $(X_1+\cdots+X_n)^N \cdot f$ are positive}. 

Powers and Reznick \cite{PR}   gave  an explicit bound for  the number $N$, and applied it to give a constructive version of Handelman's Positivstellensatz.  More explicitly, let $P\subseteq \mathbb R^n$ be a convex, compact polyhedron with non-empty interior, bounded by linear polynomials $L_1, \cdots, L_m\in \mathbb R[X]$. By choosing the sign of the $L_i$'s, we may assume that 
\begin{equation}\label{polyhedron}
P=\{x\in \mathbb R^n| L_i(x)\geq 0, ~ i=1,\cdots,m\}.
\end{equation}

\begin{theorem}[{Handelman's  Positivstellensatz, \cite{Ha}}] \label{thr-Handelman} For a polynomial $f\in \mathbb R[X]$, if $f(x)>0$ for all $x\in P$, then $f$ can be represented as 
$$f = \sum_{|\alpha|\leq M} f_\alpha L_1^{\alpha_1}\cdots L_m^{\alpha_m} $$
 for some $M\in \mathbb N$ and $f_\alpha\geq 0$ for all $\alpha=(\alpha_1,\cdots,\alpha_m) \in \mathbb N^m$ such that $|\alpha|\leq M$.
\end{theorem}
 J.L. Krivine \cite{K}   proved Handelman's Positivstellensatz for a special polyhedron. Moreover, one can find a   generalization of this Positivstellensatz in \cite[Theorem 5.4.6]{PreD} (or \cite[Theorem 7.1.6]{M}).  A bound for the number $M$  was   given by Powers and Reznick \cite{PR}, using the bound for the number $N$ in P\'{o}lya's Positivstellensatz. 

Theorem \ref{thr-Handelman} yields the following consequence.

\begin{corollary} For a polynomial $f\in \mathbb R[X]$, if $f(x)>0$ for all $x\in P$, then $f$ can be represented as 
$$f = \sum_{e=(e_1,\cdots,e_m)\in \{0,1\}^m}  f^2_e L_1^{e_1}\cdots L_m^{e_m}, $$
 where   $f_e\in \mathbb R[X]$  and $\deg(f_e^2) \leq M$.
\end{corollary}

This corollary is  a special case of Schm\"{u}dgen's Positivstellensatz  \cite{Schm91} for convex, compact polyhedra  which includes an explicit bound on the degrees of sums of  squares coefficients $f_e^2$.

Schm\"{u}dgen's Positivstellensatz has many important applications, especially in solving polynomial optimization problems and moment problems for compact semi-algebraic sets. Therefore, as a special case of Schm\"{u}dgen's Positivstellensatz,   Handelman's theorem for polynomials plays an important role in application. 

  A matrix version of  P\'{o}lya's Positivstellensatz was given by Scherer and Hol \cite{SchH}, with  applications e.g. in robust polynomial semi-definite programs.  Schm\"udgen's theorem for operator polynomials has been discovered  by Cimprič and Zalar \cite{CZ}. Positivstellens\"atze for polynomial matrices have been studied by some other authors, see for example in \cite{Ci11}, \cite{Ci12}, \cite{SaSch}, \cite{Le}.     \textit{The main aim of this paper is to give a version of Handelman's Positivstellensatz for polynomial matrices with an explicit degree bound. }

We need to introduce some notations.  For $t\in \mathbb N^*$, let  $\mathcal{M}_t(R)$  denote the ring of square matrices of order $t$ with entries from  a commutative unital ring  $R$. Denote by $\mathcal{S}_t(R)$ the subset of  $\mathcal{M}_t(R)$  consisting of all symmetric matrices. \\
In this paper we consider mainly $R$ to be the   ring  $\mathbb R[X]$   of polynomials in $n$ variables $X_1,\cdots, X_n$ with real coefficients.  Each element $\bold{A} \in  \mathcal{M}_t(\mathbb R[X])$ is a matrix whose entries are polynomials from $\mathbb R[X]$, called a \emph{polynomial matrix}. Each element $\bold{A}\in \mathcal{M}_t(\mathbb R[X])$ is also called a \emph{matrix polynomial},  because it can be viewed as a polynomial in $X_1,\cdots, X_n$ whose entries from $\mathcal{M}_t(\mathbb R)$. Namely, we can write $\bold{A}$ as 
$$  \bold{A}=\sum_{|\alpha|=0}^{d} \bold{A}_{\alpha}X^{\alpha}, $$
where $\alpha=(\alpha_1,\cdots,\alpha_n) \in \mathbb N^{n}$, $|\alpha|:=\alpha_1+\cdots + \alpha_n$, $X^{\alpha}:=X_1^{\alpha_1}\cdots X_n^{\alpha_n}$, $\bold{A}_\alpha \in \mathcal{M}_t(\mathbb R)$, $d$ is the maximum over all degree of entries of $\bold{A}$, and  it is called the \textit{degree of the matrix polynomial $\bold{A}$}.  To unify notation, throughout the paper each element of $\mathcal{M}_t(\mathbb R[X])$ is called a \emph{polynomial matrix}. \\
For any polynomial matrix $\bold{A}\in \mathcal{M}_t(\mathbb R[X])$ and for any subset $K\subseteq \mathbb R^n$,  by  $\bold{A}\succcurlyeq 0 $ (resp. $\bold{A}\succ 0$)  \textit{on $K$} we mean that for any $x\in K$, the matrix  $\bold{A}(x)$ is \emph{positive semidefinite} (resp. \textit{positive definite}),  i.e.   all eigenvalues of the matrix $\bold{A}(x)$ are non-negative (resp. positive).   \\
 For any polynomial matrices $\bold{A}, \bold{B}\in \mathcal{M}_t(\mathbb R[X])$, the notation \textit{$\bold{A}\succcurlyeq  \bold{B}$ on $K$} means that $\bold{A}-\bold{B}\succcurlyeq  \bold{0}$ on $K$.
 
Suppose that we have a convex, compact polyhedron $P\subseteq \mathbb R^n$  with non-empty interior, bounded by linear polynomials $L_1, \cdots, L_m\in \mathbb R[X]$, defined by (\ref{polyhedron}). Let $\bold{F} \in \mathcal{S}_t(\mathbb R[X])$ be a  polynomial matrix of degree $d>0$. Assume $\bold{F} \succ \bold{0}$ on $P$. The main result of this paper is presented in     Theorem \ref{thr:main} which is a  matrix version of Handelman's Positivstellensatz, stating  that there exists a  number $N_0$  such that for all integer $N >N_0$ the polynomial matrix $\bold{F}$ can be written as
$$\bold{F}=\sum_{|\alpha|=N+d}\bold{F}_\alpha L_1^{\alpha_1}\cdots L_m^{\alpha_m},$$
where  $\bold{F}_\alpha \in \mathcal{S}_t(\mathbb R)$ are positive definite scalar  matrices  with $|\alpha|=N+d$.\\
 The main idea  in the proof of this theorem inherits  from Powers and Reznick   \cite{PR}, using a matrix version of P\'{o}lya's Positivstellensatz  \cite{SchH} and the continuity of  eigenvalue functions of the polynomial matrix $\bold{F}$ on the entries  of $\bold{F}$ (by \cite[Theorem 1]{Z}). 
  As a corollary of this theorem, we give a special case of Schm\"udgen's Positivstellensatz for polynomial matrices positive definite on  convex, compact polyhedra.  Furthermore, we give a procedure to find such a representation for the polynomial matrix $\bold{F}$. 
 \section{Representation of polynomial matrices positive definite on simplices}
In this section we consider a simple case  where $P$ is an $n$-simplex with vertices 
$\{v_0,v_1,\cdots,v_n\}$ and let $\{L_0,L_1,\cdots,L_n\}$ be the set of barycentric coordinates of $P$, that is, each $L_i\in \mathbb R[X]$ is linear and 
\begin{equation}\label{simplex}
X=\sum_{i=0}^n L_i(X) v_i, ~ \sum_{i=0}^nL_i(X)=1, L_i(v_j)=\delta_{ij}.
\end{equation} 
Let $\bold{F}\in \mathcal{S}_t(\mathbb R[X])$ be a polynomial matrix of degree $d>0$. We can express   $\bold{F}$ as 
$$\bold{F}(X)=\sum_{|\alpha|\leq d} \bold{A}_\alpha X^{\alpha}, $$
where $\bold{A}_\alpha \in \mathcal{M}_t(\mathbb R)$. 

Let us consider the \textit{Bernstein-B\'{e}zier form} of $\bold{F}$ with respect to $P$:
\begin{equation}\label{Bernstein-Bezier}
\bold{\widetilde{F}}_d(Y):=\bold{\widetilde{F}}_d(Y_0,\cdots,Y_n):=\sum_{|\alpha|\leq d} \bold{A}_\alpha \Big(\sum_{i=0}^n Y_iv_i\Big)^\alpha \Big(\sum_{i=0}^n Y_i\Big)^{d-|\alpha|}.
\end{equation}
It is easy to see that $\bold{\widetilde{F}}_d(Y)\in \mathcal{S}_t(\mathbb R[Y])$ is a homogeneous polynomial matrix of degree $d$. Moreover, it follows from the relations (\ref{simplex}) that 
$$ \bold{\widetilde{F}}_d(L_0,\cdots,L_n)=\bold{F}(X).$$
Following Scherer and Hol \cite{SchH}, for each multi-index $\alpha=(\alpha_1,\cdots,\alpha_n)\in \mathbb N^n$, let us  denote
$$  \alpha!:=\alpha_1!\cdots \alpha_n!; ~ D_\alpha:=\partial_1^{\alpha_1}\cdots\partial_n^{\alpha_n}.$$
With these notations,  we can re-write $\bold{F}$ as 
$$\bold{F}(X)= \sum_{|\alpha|\leq d} \dfrac{D_\alpha \bold{F}(0)}{\alpha!}X^\alpha.$$
With the spectral norm $\|\bold{\cdot}\|$, following Scherer and Hol \cite{SchH}, we define
\begin{equation}\label{pt:norm}
C(\bold{F}):=\max_{|\alpha|\leq d}\dfrac{\|D_\alpha \bold{F}(0)\|}{|\alpha|!}.
\end{equation}

Using these notations,  we have the following representation of polynomial matrices which are positive on simplices.
\begin{theorem} \label{thr-simplex} Let $P\subseteq \mathbb R^n$ be an $n$-simplex given as above and $\bold{F}\in \mathcal{S}_t(\mathbb R[X])$ a polynomial matrix of degree $d>0$. Assume that $\bold{F}\succcurlyeq \lambda \bold{I}_t$ on $P$ for some $\lambda >0$.  Let  $C:=C(\bold{\widetilde{F}_d})$. Then for each $N>\dfrac{d(d-1)}{2} \dfrac{C}{\lambda}-d$, $\bold{F}$ can be represented as
$$\bold{F}=\sum_{|\alpha|=N+d}\bold{F}_\alpha L_0^{\alpha_0}\cdots L_n^{\alpha_n},$$
where  each  $\bold{F}_\alpha\in \mathcal{S}_t(\mathbb R)$ is positive  definite.
\end{theorem}

\begin{proof} Let us denote by $\Delta_{n+1}$ the standard simplex in $\mathbb R^{n+1}$, i.e.
$$ \Delta_{n+1}=\{(y_0,\cdots,y_n)\in \mathbb R^{n+1}| y_i\geq 0, \sum_{i=0}^n y_i=1\}. $$
Since $\bold{F}(x)\succcurlyeq \lambda \bold{I}_t$  for all $x\in P$, the  Bernstein-B\'{e}zier form
 $\bold{\widetilde{F}_d}$ of $\bold{F}$ with respect to $P$ satisfies
$$ \bold{\widetilde{F}_d}(y_0,\cdots,y_n)\succcurlyeq \lambda \bold{I}_t, \forall (y_0,\cdots,y_n)\in \Delta_{n+1}. $$
Then it follows from P\'{o}lya's theorem for polynomial matrices \cite[Theorem 3]{SchH}, that for each $N> \dfrac{d(d-1)}{2} \dfrac{C}{\lambda}-d$,  
\begin{equation}\label{pt-cm-simplex}
(\sum_{i=0}^n Y_i)^N\bold{\widetilde{F}_d}(Y)=\sum_{|\alpha|=N+d}\bold{F}_\alpha Y_0^{\alpha_0}\cdots Y_n^{\alpha_n},
\end{equation}
where  each  $\bold{F}_\alpha\in \mathcal{S}_t(\mathbb R)$ is positive  definite. Substituting $Y_i$ by $L_i$ on both sides of (\ref{pt-cm-simplex}), noting that 
$$\bold{\widetilde{F}_d}(L_0(X),\cdots,L_n(X)) = F(X) \mbox{ and } \sum_{i=0}^N L_i(X)=1,$$
 we obtain the required representation for $\bold{F}$.
\end{proof}
\section{Representation of polynomial matrices positive definite on convex, compact polyhedra}
Throughout this section, let $P\subseteq \mathbb R^n$ be a  convex, compact polyhedron  with non-empty interior, given by (\ref{polyhedron}). 
By \cite{Schw0},  there exist positive numbers $c_i\in \mathbb R$ such that $\displaystyle\sum_{i=1}^m c_iL_i(X)=1$. Replacing each  $L_i$ by $c_iL_i$ we  may assume that 
\begin{equation}\label{quanhe-lambda} \sum_{i=1}^m L_i(X)=1. 
\end{equation}
Moreover, it is easy to check that for each  $i=1,\cdots,n$,  there exist real numbers $b_{ij}\in \mathbb R, j=1,\cdots,m$ such that 
$$ X_i = \sum_{j=1}^m b_{ij}L_j(X). $$
Let us consider the $n\times m$ matrix $\bold{B}:=(b_{ij})_{i=1,\cdots,n; j=1,\cdots,m}$. Then for $X=(X_1,\cdots,X_n)$ and  $L=(L_1,\cdots,L_m)$, we have 
$ X^T=\bold{B}\cdot L^T. $
In other words, we have 
\begin{equation}\label{eq:matran}X=L\cdot \bold{B}^T.
\end{equation} 
Denote $\mathbb R[Y]:=\mathbb R[Y_1,\cdots,Y_m]$, and consider the ring homomorphism 
$$\varphi: \mathbb R[Y] \rightarrow \mathbb R[X], \quad Y_i\longmapsto L_i(X), \forall i=1,\cdots,m.$$
It follows from (\ref{quanhe-lambda}) that $\sum_{i=1}^m Y_i-1\in \Ker(\varphi)$. Hence we may assume that the ideal $I:=\mbox{Ker}(\varphi)$ is generated by polynomials $R_1(Y),\cdots, R_s(Y)\in \mathbb R[Y]$,
 $$I:=\mbox{Ker}(\varphi)=\seq{R_1(Y),\cdots,R_s(Y)},$$
 where $\sum_{i=1}^m Y_i-1$  is one of the $R_i$'s. 
 
 Note that the homomorphism $\varphi$ induces a ring homomorphism 
$$ M_\varphi: \mathcal{M}_t(\mathbb R[Y]) \longrightarrow \mathcal{M}_t(\mathbb R[X]),\quad  \bold{G}=(g_{ij}(Y))\longmapsto (\varphi(g_{ij}(Y))).   $$

\begin{lemma}\label{lm:surjective} The homomorphism $M_\varphi$ is surjective, and 
$$ \mathcal{I}:=\Ker(M_\varphi)=\seq{R_1(Y)\bold{I}_t,\cdots, R_s(Y)\bold{I}_t},$$
where $\bold{I}_t$ denotes the identity matrix in $\mathcal{M}_t(\mathbb R[Y])$.
\end{lemma}
\begin{proof}  For each $g(X)=\sum_{|\alpha|\leq d}a_\alpha X^\alpha\in \mathbb R[X]$, denote 
\begin{equation}\label{pt:tilde}
\widetilde{g}(Y):=\sum_{|\alpha|\leq d}a_\alpha (Y\cdot B^T)^\alpha \big(\sum_{i=1}^mY_i\big)^{d-|\alpha|}\in \mathbb R[Y].
\end{equation}
It is clear that $\widetilde{g}$ is homogeneous of degree $d$. Moreover  $\varphi(\widetilde{g}(Y))=g(X)$. Hence $\varphi$ is surjective.  Then the surjectivity of $M_\varphi$ follows from that of $\varphi$.  

On the other hand, $\bold{G}=(g_{ij}(Y))\in \Ker(M_\varphi)$ if and only if $g_{ij}\in \Ker(\varphi)$ for all $i,j=1,\cdots,t$. Hence for each $i,j=1,\cdots,t$ we have
 $$g_{ij}(Y)=\sum_{k=1}^s a_{ijk}(Y)R_k(Y), \mbox{ for some } a_{ijk}(Y) \in \mathbb R[Y]. $$
Then $\bold{G}$ can be written as
$$\bold{G}=\sum_{k=1}^s R_k\bold{A_k}=\sum_{k=1}^s (R_k\bold{I}_t)\bold{A_k}, $$
where  $\bold{A_k}=(a_{ijk}(Y))\in \mathcal{M}_t(\mathbb R[Y])$ for each $k=1,\cdots, s$. It is equivalent to the fact that $\bold{G}\in \seq{R_1\bold{I}_t,\cdots, R_s\bold{I}_t}.$
The proof is complete. 
\end{proof}

 Let $\bold{F}=(f_{ij})\in \mathcal{S}_t(\mathbb R[X])$ be a polynomial matrix of degree $d>0$. Denote $\bold{\widetilde{F}}:=(\widetilde{f_{ij}})\in \mathcal{S}_t(\mathbb R[Y])$, where each $\widetilde{f_{ij}}$ is defined by (\ref{pt:tilde}), which is a homogeneous polynomial of degree $d$.

Assume  $\lambda(\bold{F})$ is an eigenvalue function of $\bold{F}$. It follows from \cite[Theorem 1]{Z} that $\lambda(\bold{F})$  is a continuous function on $f_{ij}(X)$, $i,j=1,\cdots,t$. That is, there exists a continuous funtion $\Lambda : \mathbb R^{t\times t} \rightarrow \mathbb R$ such that $\lambda(\bold{F})=\Lambda(f_{ij}(X))$.  Denote 
$\widetilde{\lambda(\bold{F})}(Y):=\Lambda(\widetilde{f_{ij}}(Y))$, which is actually an eigenvalue function of the polynomial matrix $\bold{\widetilde{F}}$.  

 Denote $R(Y):=\displaystyle \sum_{i=1}^s R_i^2(Y)$. With the notations given above, we have the following useful lemma.

\begin{lemma}\label{lm:minors} Let $\bold{F}=(f_{ij})\in \mathcal{S}_t(\mathbb R[X])$ be a polynomial matrix of degree $d>0$.  Let $\lambda(\bold{F})$ is an eigenvalue function of $\bold{F}$.    If $\lambda(\bold{F})> 0$ on $P$, then there exists a sufficiently large constant $c$ such that $\widetilde{\lambda(\bold{F})}+cR>0$ on the standard $m$-simplex $ \Delta_m$.  More explicitly, this holds for $c>-m_1/m_2$, where $m_1$ is the minimum of $\widetilde{\lambda(\bold{F})}$ on $\Delta_m$ and $m_2$ is the minimum of the polynomial $R$ on the compact set 
$\Delta_m\cap \{y\in \mathbb R^m|\widetilde{\lambda(\bold{F})}(y)\leq 0\}$. 
\end{lemma}
\begin{proof}
The proof goes along the same lines as the proof of \cite[Lemma 4]{PR}, using continuity of the function $\widetilde{\lambda(\bold{F})}$.
\end{proof}

Applying this lemma, we have 

\begin{lemma}\label{lm:positive-on-simplex} Let $\bold{F}=(f_{ij})\in \mathcal{S}_t(\mathbb R[X])$ be a polynomial matrix of degree $d>0$. Denote 
$\bold{\widetilde{F}}:=(\widetilde{f_{ij}})\in \mathcal{S}_t(\mathbb R[Y])$. 
 Assume $\bold{F}\succ \bold{0}$ on $ P$. Then there exists a sufficiently large constant  $c$ such that $\widetilde{\bold{F}}+cR\bold{I}_t \succ \bold{0}$ on the standard $m$-simplex  $\Delta_m$.
\end{lemma}

\begin{proof}  Since $\bold{F}$ is positive definite on $P$, its eigenvalue functions
$\lambda_k(\bold{F}), k=1,\cdots,t$, are positive   on $P$. It follows from Lemma \ref{lm:minors} that for each $k$, there exists a sufficiently large constant $c_k$ such that $\widetilde{\lambda_k(\bold{F})} + c_k R$ is positive on $\Delta_m$. Let $ c:=\displaystyle\max_{k=1,\cdots,t}{c_k}$. Then $\widetilde{\lambda_k(\bold{F})} + c R$ is positive on $\Delta_m$ for each $k=1, \cdots, t$. Note that, $\widetilde{\lambda_k(\bold{F})}$, $k=1,\cdots,t$, are eigenvalues of the polynomial matrix $\widetilde{\bold{F}}$. Moreover, the eigenvalues of the matrix $\widetilde{\bold{F}}+cR\bold{I}_t$ are $\widetilde{\lambda_k(\bold{F})}+cR$, $k=1,\cdots,t$. It follows that  $\widetilde{\bold{F}}+cr\bold{I}_t$  is positive definite on $\Delta_m$. The proof is complete. 
\end{proof}

Note that $\overline{\bold{F}}:=\widetilde{\bold{F}}+cR\bold{I}_t$ need not be homogeneous. However, by homogenization $\overline{\bold{F}}$ by $\displaystyle\sum_{i=1}^m Y_i$, we obtain a homogeneous polynomial matrix of the same degree as $\overline{\bold{F}}$. More explicitly, if we express $\overline{\bold{F}}$ as 
$$\overline{\bold{F}} = \sum_{|\beta|\leq d} \bold{\overline{F}}_\beta Y^\beta,\quad \bold{\overline{F}}_\beta\in \mathcal{S}_t(\mathbb R), $$
then its homogenization by $\displaystyle\sum_{i=1}^m Y_i$ is 
\begin{equation}\label{pt:homogenization} \overline{\bold{F}}^h=\sum_{|\beta|\leq d} \bold{\overline{F}}_\beta Y^\beta (\sum_{i=1}^m Y_i)^{d-|\beta|}.
\end{equation}
  $ \overline{\bold{F}}^h$ is a homogeneous   polynomial matrix of degree $d$. Moreover, 
$M_\varphi(\overline{\bold{F}}^h)=\bold{F}$, and $\overline{\bold{F}}^h$ is positive definite on $\Delta_m$.

 Now we can state and prove the following matrix version of Handelman's Positivstellensatz.

\begin{theorem}\label{thr:main} Let $P$, $\bold{F}$, $\overline{\bold{F}}$, $\overline{\bold{F}}^h$ be given as above, with $\bold{F}$  positive definite on $P$. Assume that  $\overline{\bold{F}}^h\succcurlyeq \lambda \bold{I}_t$ on $\Delta_m$ for some $\lambda >0$.  Let  $d:=\mbox{deg}(\overline{\bold{F}})$ and $C:=C(\overline{\bold{F}}^h)$.  Then for each $N> \dfrac{d(d-1)}{2} \dfrac{C}{\lambda}-d$, $\bold{F}$ can be represented as
\begin{equation}\label{equ-representation}
\bold{F}=\sum_{|\alpha|=N+d}\bold{F}_\alpha L_1^{\alpha_1}\cdots L_m^{\alpha_m},
\end{equation}
where  each  $\bold{F}_\alpha\in \mathcal{S}_t(\mathbb R)$ is positive  definite. 
\end{theorem}
\begin{proof}
Firstly, applying  the matrix version of P\'{o}lya's Positivstellensatz given in  \cite[Theorem 3]{SchH} for $\overline{\bold{F}}^h$, observing that $d=\deg(\overline{\bold{F}}^h)$. Then, applying $M_\varphi$, using the fact that $M_\varphi(\overline{\bold{F}}^h)=\bold{F}$ and $\varphi(\displaystyle\sum_{i=1}^m Y_i)=1$.
\end{proof}

As a summary, we formulate the construction given  above  as a procedure  to find a   representation  
for the  polynomial matrix $\bold{F}=(f_{ij})\in \mathcal{S}_t(\mathbb R[X])$ positive definite on  a convex, compact polyhedron $P\subseteq \mathbb R^n$ as follows:
 
\begin{itemize}
\item[(1)] Following \cite{Ha} to find positive constants $c_i\in \mathbb R$ such that $\sum_{i=1}^m c_iL_i(X) = 1$. Constructing the $c_i$'s comes down to find a positive solution to an under-determined linear system.
\item[(2)] Solving the system of equations 
$$X_i=\sum_{j=1}^m b_{ij}L_{i}(X), \quad i=1,\cdots,n, $$
to find the matrix $\bold{B}=(b_{ij})_{i=1,\cdots,n; j=1,\cdots,m}$.
\item[(3)] Using (\ref{pt:tilde}) to find $\widetilde{f_{ij}}$, $i,j=1,\cdots,t$.
\item[(4)] Using  Gr\"{o}bner bases to find a basis $\{R_1,\cdots,R_s\}$ for the kernel $\mbox{Ker}(\varphi)$ of the ring homomorphism $\varphi$.
\item[(5)] Following the proof of Lemma \ref{lm:positive-on-simplex} to find a sufficiently large $c$ such that $\bold{\widetilde{F}}+c R\bold{I}_t \succ \bold{0}$ on $\Delta_m$.
\item[(6)] Using (\ref{pt:homogenization}) to construct the homogenization $\overline{\bold{F}}^h$ of $\overline{\bold{F}}:=\bold{\widetilde{F}}+c  R \bold{I}_t$.
\item[(7)] Following the proof of  Lemma \ref{lm:compact} below to find the positive number $\lambda$ such that $\overline{\bold{F}}^h(y)\succcurlyeq \lambda \bold{I}_t$ for all $y\in \Delta_m$.
\end{itemize}
\begin{lemma} \label{lm:compact}Let $K\subseteq \mathbb R^m$ be a non-empty compact set, and $\bold{G}\in \mathcal{S}_t(\mathbb R[Y])$. Then there exists a number $c \in \mathbb R$ such that 
$$\bold{G}(y)\succcurlyeq   c\bold{I}_t, \mbox{ for all } y \in K.  $$
In particular, if $\bold{G}(y)\succ 0$ for all $y\in K$, then we can choose a number  $c >0 $ such that  $\bold{G}(y)\succcurlyeq  c \bold{I}_t, \mbox{ for all } y \in K.  $
\end{lemma}
\begin{proof} Let $\lambda_{1}(\bold{G}), \cdots, \lambda_t(\bold{G})$ be (real-valued) eigenvalue functions of the polynomial matrix $\bold{G}\in \mathcal{S}_t(\mathbb R[Y]). $ It follows from \cite[Theorem 1]{Z} that $\lambda_i(\bold{G})$ are continuous functions. Since $K$ is compact, let 
$$c_i:= \min_{y\in K} \lambda_i(\bold{G})(y), \quad i=1,\cdots, t.$$
Denote $c:=\min_{i=1,\cdots,t} c_i$. Since eigenvalue functions of $\bold{G}-c\bold{I}_t$ are $\lambda_i(\bold{G})-c$, $i=1,\cdots,t$, it follows from the definition of $c$ that 
$$ \lambda_i(\bold{G})(y) - c \geq \lambda_i(\bold{G})(y) -c_i\geq 0$$
for all $y\in K$ and for all $i=1,\cdots,t$. This implies that 
$\bold{G}(y)\succcurlyeq   c\bold{I}_t, \mbox{ for all } y \in K.$
\end{proof}
\begin{itemize}
\item[(8)] Using the formula (\ref{pt:norm}) to find the number $C:=C(\overline{\bold{F}}^h)$.
\item[(9)] Find a number $N> \dfrac{d(d-1)}{2} \dfrac{C}{\lambda}-d$.
\item[(10)] Find the matrix coefficients of the polynomial matrix $(\sum_{i=1}^m Y_i)^N \overline{\bold{F}}^h\in \mathcal{S}_t(\mathbb R[Y])$, substituting $Y_i$ by $L_i(X)$, we obtain the desired representation for $\bold{F}$.
\end{itemize}

We illustrate the procedure given above by the following example which is computed explicitly using MATLAB Version 7.10 (Release 2010a) and its add-on GloptiPoly 3 discovered by Henrion, Lasserre and L\"ofberg \cite{HLL}. 

\begin{example}\rm  Let us consider the unit square centered at the origin 
$$P:=\{(x,y)\in \mathbb R^2| L'_1=1+x\geq 0, L'_2=1-x\geq 0, L'_3=1+y\geq 0, L'_4=1-y\geq 0\}. $$
Choosing $c_1=c_2=c_3=c_4=\dfrac{1}{4}$, we have $\sum_{i=1}^4c_iL'_i(x,y)=1$. Therefore, consider 
$$L_1:=\dfrac{1}{4}+\dfrac{1}{4}x, ~L_2:=\dfrac{1}{4}-\dfrac{1}{4}x,~ L_3:=\dfrac{1}{4}+\dfrac{1}{4}y, ~L_4:=\dfrac{1}{4}-\dfrac{1}{4}y \in \mathbb R[x,y],$$
 we have $ \sum_{i=1}^4 L_i = 1$.
 
 It is easy to see that the matrix $\bold{B}=\left[\begin{array}{cccc}
 2&-2&0&0\\
 0&0&2&-2
 \end{array}\right] $ satisfies the equation 
 $$\bold{B}\cdot [L_1 ~ L_2 ~ L_3 ~ L_4]^T = [x ~ y]^T.$$

Let $\varphi: \mathbb R[y_1,y_2,y_3,y_4] \rightarrow \mathbb R[x,y]$ be the ring homomorphism defined by $\varphi(y_i):=L_i(x,y)$, $i=1,2,3,4$. Using any monomial ordering in $\mathbb R[y_1,y_2,y_3,y_4]$ we can find a Gr\"{o}bner basis for the kernel $\mbox{Ker}(\varphi)$ of $\varphi$:
$$\{R_1, R_2\}:=\{y_1+y_2-\frac{1}{2}, y_3+y_4-\frac{1}{2}\}.$$
Consider $R:=R_1^2+R_2^2$.

Now we  consider the polynomial matrix $$\bold{F}:=\begin{bmatrix}
-4x^2y+7x^2+y+3&x^3+5xy-3x\\
x^3+5xy-3x&x^4+x^2y+3x^2-4y+6\\
\end{bmatrix}.$$
Eigenvalue functions of $\bold{F}$ are 
$$ \lambda_1(\bold{F})=6x^2 - 4x^2y - 4y + 6; ~  \lambda_2(\bold{F})=x^4 + x^2y + 4x^2 +y+ 3.$$
For any $(x,y)\in P$ we have $\lambda_i(\bold{F})(x,y)\geq 2$, $i=1,2$. Hence $\bold{F}(x,y)\succcurlyeq 2\bold{I}_2$ for every $(x,y)\in P$.\\
With the matrix $\bold{B}$ considered above, using the formula (\ref{pt:tilde}), we   find   $\widetilde{f_{ij}}$, $i,j=1,2$, and then we obtain  the polynomial matrix  $\widetilde{\bold{F}}=(\widetilde{f}_{ij})$. We can compute exactly the eigenvalue functions $\lambda_1(\widetilde{\bold{F}})$ and $\lambda_2(\widetilde{\bold{F}})$ of $\widetilde{\bold{F}}$ which satisfy  
 $$\min_{\Delta_4}\lambda_1(\widetilde{\bold{F}}) =1,~ \min_{\Delta_4}\lambda_2(\widetilde{\bold{F}}) =-2.$$
Moreover, $\displaystyle \min_{\Delta_4\cap \{\lambda_2(\widetilde{\bold{F}})\leq 0\}} R(y_1,y_2,y_3,y_4) = 0.125$.  Thus we can choose 
$$c>-\dfrac{-2}{0.125}=16, \mbox{ namely}, ~c=17,$$
 for which $\overline{\bold{F}}:=\widetilde{\bold{F}}+cR \bold{I}_2 \succ \bold{0} $ on $\Delta_4$.\\
Homogenizing  $\overline{\bold{F}}$ by $\sum\limits_{i=1}^4y_i$ we obtain a homogeneous polynomial matrix  $\overline{\bold{F}}^h=(\overline{f_{ij}}^h)$. Then we compute exactly the eigenvalue functions of the matrix $\overline{\bold{F}}^h$ which satisfy 
$$\min_{\Delta_4}\lambda_1(\overline{\bold{F}}^h) = 1.9706,~ \min_{\Delta_4}\lambda_2(\overline{\bold{F}}^h) = 1.5294.$$
It follows that $\overline{\bold{F}}^h\succcurlyeq  1.5294 ~  \bold{ I}_2$ on $\Delta_4$, and put $\lambda:=1.5294$.\\
Using the formula (\ref{pt:norm}), we can find the number $C:=C(\overline{\bold{F}}^h)=\dfrac{1044}{24}=\dfrac{87}{2}$. \\
Therefore,  choosing $N=167$, the polynomial matrix $(y_1+y_2+y_3+y_4)^{167}\overline{\bold{F}}^h$ has positive definite coefficients.\\
Find the matrix coefficients of the polynomial matrix $(y_1+y_2+y_3+y_4)^{167}\overline{\bold{F}}^h\in \mathcal{S}_t(\mathbb R[y_1,y_2,y_3,y_4])$, substituting $y_i$ by $L_i(x,y)$, we obtain the desired representation for $\bold{F}$.
\end{example}

As a consequence of Theorem \ref{thr:main}, we obtain the following matrix version  of Schm\"udgen's Positivstellensatz for convex, compact polyhedra.
\begin{corollary}
Let $P$,  $\bold{F}$, $\overline{\bold{F}}$, $\overline{\bold{F}}^h$ be given as above, with $\bold{F}$  positive definite on $P$. Assume that  $\overline{\bold{F}}^h\succcurlyeq \lambda \bold{I}_t$ on $\Delta_m$ for some $\lambda >0$. Let  $d:=\mbox{deg}(\overline{\bold{F}})$ and $C:=C(\overline{\bold{F}}^h)$.  Then for $N> \dfrac{d(d-1)}{2} \dfrac{C}{\lambda}-d$, $\bold{F}$ can be represented as
\begin{equation}\label{coro-equ-representation}
\bold{F}=\sum_{e=(e_1,\cdots,e_m)\in \{0,1\}^m}(\bold{F}_e^T \bold{F}_e) L_1^{e_1}\cdots L_m^{e_m},
\end{equation}
where   $\bold{F}_e\in \mathcal{M}_t(\mathbb R[X])$  and the degree of each sum of squares  $\bold{F}_e^T\bold{F}_e$ does not exceed $N+d$.
\end{corollary}
\pf The proof follows directly from Theorem \ref{thr:main}, with the observation that any positive definite matrix $\bold{F_\alpha}\in \mathcal{S}_t(\mathbb R)$ can be written as 
$$ \bold{F_\alpha} = \bold{G_\alpha}^T \bold{G_\alpha}, $$
where  $ \bold{G_\alpha}\in \mathcal{M}_t(\mathbb R)$ is a non-singular matrix.

\epf 
\textbf{Acknowledgements} ~  The authors would like to thank the anonymous referees for their useful comments and suggestions. They would also like to   thank  Dr. Ngo Lam Xuan Chau for his fruitful discussion to compute the kernel of the homomorphism $M_\varphi$.  Both  authors are partially supported by  the Vietnam National Foundation for Science and Technology Development (NAFOSTED) under   grant number  101.01-2016.27.



\begin{thebibliography}{99}

\bibitem  {Ci11}  Cimprič, J.:   Strict Positivstellens\"atze for matrix polynomials with scalar constraints.  Linear Algebra and its Applications \textbf{ 434}(8), 1879-1883 (2011)
\bibitem  {Ci12} Cimprič, J.:  Real algebraic geometry for matrices over commutative rings.  Journal of Algebra \textbf{359},  89-103 (2012)
\bibitem  {CZ} Cimprič, J., Zalar, A.:  Moment problems for operator polynomials.  J. Math. Anal. Appl.  \textbf{401}(1), 307-316 (2013)
\bibitem  {Ha} Handelman, D.:  Representing polynomials by positive linear functions on compact convex polyhedra.    Pacific J. Math. \textbf{132}, 35-62 (1988)
\bibitem{HLL} Henrion, D.,  Lasserre, J., Loefberg, J.:  GloptiPoly 3: moments, optimization and semidefinite programming. \url{http://homepages.laas.fr/henrion/software/gloptipoly3/} 
\bibitem{K}  Krivine, J.L.:  Quelques propértiés des préordres dans les anneaux commutatifs unitaires.   C. R. Acad. Sci. Paris \textbf{258}, 3417-3418 (1964)
\bibitem{Le} Lê, C.-T.: Some Positivstellensätze for polynomial matrices.  Positivity \textbf{19}(3), 513-528 (2015)
\bibitem{M} Marshall, M.:  Positive polynomials and sums of squares.  Math. Surveys and Monographs, Amer. Math. Soc., Providence, R.I. (2008)
\bibitem{Po}   P\'{o}lya, G.:   \"{U}ber positive Darstellung von Polynomen Vierteljschr.   Naturforsch. Ges. Z\"{u}rich \textbf{73}, 141-145 (1928), in: R.P. Boas (Ed.), Collected Papers Vol. 2, MIT Press, Cambridge, MA,  pp. 309–313 (1974)
\bibitem{PR} Powers, V.,  Reznick, B.:   A new bound for P\'{o}lya's theorem with applications to polynomials positive on polyhedra.  J. Pure Appl. Algebra  \textbf{164}, 221-229 (2001)
\bibitem{PreD} Prestel, A.,   Delzell, C.N.:  Positive Polynomials - from Hilbert’s 17th Problem to Real Algebra. Springer-Verlag, New York (2001)
\bibitem{SaSch} Savchuk, Y., Schm\"udgen, K.:  Positivstellens\"atze for algebras of matrices. 
Linear Algebra Appl. \textbf{43 },  758-788 (2012)
\bibitem {SchH} Scherer, C.W.,  Hol, C.W.J.:   Matrix sum-of-squares relaxations for robust semi-definite programs. Math. Program. \textbf{107} no. 1-2, Ser. B, 189–211 (2006)
\bibitem  {Schm91} Schm\"{u}dgen, K.:   The K-moment problem for compact semi-algebraic sets.  Math. Ann. \textbf{289} (2), 203–206 (1991)
\bibitem   {Schw0} Schweighofer, M.:   An algorithmic approach to Schm\"{u}dgen's Positivstellensatz.  J. Pure Appl. Algebra \textbf{166}(3), 307–319 (2002)
\bibitem   {Z} Zedek, M.:  Continuity and Location of Zeros of Linear Combinations of Polynomials.  Proc. Amer. Math. Soc. \textbf{16}, 78–84 (1965)

\end{thebibliography}
\end{document}